\documentclass[a4paper,numbook,babel,final,envcountsame,11pt]{article}

\usepackage{amssymb}
\usepackage{amsthm}
\usepackage{amsmath}
\usepackage{graphicx}
\usepackage{array,hhline}

\newtheorem{lemma}{Lemma}[section]
\newtheorem{theorem}[lemma]{Theorem}
\newtheorem{proposition}[lemma]{Proposition}

\newtheorem{corollary}[lemma]{Corollary}
\theoremstyle{definition}
\newtheorem{definition}[lemma]{Definition}

\numberwithin{equation}{section}
\numberwithin{figure}{section}

\begin{document}

\title{\huge On the dimension of spaces of algebraic curves}     
\author{Hakop Hakopian, Harutyun Kloyan}

\date{}          

\maketitle




\begin{abstract} Let a set of nodes $\mathcal X$ in plain be $n$-independent, i.e., each node has a fundamental polynomial of degree $n.$ Suppose also that
$|\mathcal X|= d(n,k-2)+2,$ where $d(n,k-2) = (n+1)+n+\cdots+(n-k+4)$ and $\ k\le n-1.$ In this paper we prove that there can be at most $4$ linearly independent curves of degree less than or equal to $k$ passing through
all the nodes of $\mathcal X.$ We provide a characterization of the case when there are exactly four such curves. Namely, we prove that then the set $\mathcal X$ has a very special construction: All its nodes but two belong to a (maximal) curve of degree $k-2.$

At the end, an important application to the Gasca-Maeztu conjecture is provided.
\end{abstract}




\section{\label{1}  Introduction}

Denote the space of all bivariate polynomials of total degree $\leq n$ by $\Pi_n:$
$$\Pi_n= \left\{ \sum_{i+j\leq n} a_{ij} x^i y^j \right\}.$$
We have that
\begin{equation*}
  N := N_n := \dim \Pi_n = (1/2)(n+1)(n+2).
\end{equation*}

Consider a set of $s$ distinct nodes
\begin{equation*}
 {\mathcal X}={\mathcal X}_s = \{(x_1,y_1), (x_2,y_2), \ldots, (x_s,y_s) \}.
\end{equation*}
The problem of finding a polynomial $p \in \Pi_n$ which satisfies the conditions
\begin{equation}\label{eq:intpr}
  p(x_i,y_i) = c_i, \quad i=1, \ldots, s,
\end{equation}
is called interpolation problem.

A polynomial $p \in \Pi_n$ is called a fundamental polynomial for a node $A\in{\mathcal X}$ if
\begin{equation*}
 p(A)=1\ \ \text{and}\ \ p\big\vert_{{\mathcal X}\setminus\{A\}} = 0,
\end{equation*}
where $p\big\vert_{\mathcal X}$ means the restriction of $p$ on $\mathcal X.$ We denote the fundamental polynomial
by $p_{A}^\star.$ Sometimes we call fundamental also a polynomial that vanishes at all nodes
of ${\mathcal X}$ but one, since it is a nonzero constant times a fundamental polynomial.

\begin{definition} \label{poised}
The interpolation problem with a set of nodes ${\mathcal X}_s$ and $\Pi_n$
is called \emph{$n$-poised} if for any data $(c_1,\ldots, c_s)$
there is a \emph{unique} polynomial $p\in\Pi_n$ satisfying the
interpolation conditions \eqref{eq:intpr}.
\end{definition}
A necessary condition of
poisedness is $|{\mathcal X}_s|= s = N.$ 

\begin{proposition}\label{prp:int0}
A set of nodes ${\mathcal X}_N$ is $n$-poised if and only if
\begin{equation*}
 p\in\Pi_n\ \ \text{and}\ \ p\big\vert_{{\mathcal X}_N} = 0
 \quad\implies\quad p = 0.
\end{equation*}
\end{proposition}

Next, let us consider the concept of $n$-independence (see \cite{E}, \cite{HJZ}).

\begin{definition}
A set of nodes ${\mathcal X}$ is called \emph{$n$-independent} if all its
nodes have $n$-fundamental polynomials. Otherwise, it is called \emph{$n$-dependent}.
\end{definition}

Fundamental polynomials are linearly independent. Therefore a necessary condition of $n$-independence of ${\mathcal X}_s$
is $s \leq N$.

\section{\label{2}  Some properties of $n$-independent nodes}

Let us start with the following simple 
\begin{lemma}[e.g., \cite{HM} Lemma 2.2]\label{XA}
Suppose that a node set ${\mathcal X}$ is $n$-independent and a node $A\notin \mathcal X$
has $n$-fundamental polynomial with
respect to the set ${\mathcal X}\cup \{A\}.$ Then the latter
node set is $n$-independent, too.
\end{lemma} Denote the distance between the points $A$ and $B$ by $\rho(A,B).$
Let us bring the following (see e.g., \cite{H82}, \cite{HakTor})
\begin{lemma} \label{eps'} Suppose that ${\mathcal X}_s=\{A_i\}_{i=1}^s$ is an $n$-independent set. Then there is a number $\epsilon>0$ such that
any set ${\mathcal X}_s'=\{A_i'\}_{i=1}^s,$ with the property that $\rho(A_i,A_i')<\epsilon,\ i=1,\ldots,s,$  is $n$-independent too.
\end{lemma}

Next result concerns the extension of $n$-independent sets  
\begin{lemma}[e.g., \cite{HJZ},  Lemma 2.1]\label{ext}
Any $n$-independent set ${\mathcal X}$ with $|{\mathcal X}|<N$ can be enlarged to an $n$-poised set. 
\end{lemma}
In the sequel we will need the following modification of the above result.
\begin{lemma}\label{extm}
Given $n$-independent sets ${\mathcal X}_{s_i},\ i=1,\ldots,m,$ where $|{\mathcal X}_{s_i}|={s_i}<N,$ a node $A,$ and any number $\epsilon>0.$ Then there is a node $A',$ such that 
$\rho(A,A')<\epsilon$ and each set ${\mathcal X}_{s_i}\cup \{A'\},\ i=1,\ldots,m,$ is $n$-independent.
\end{lemma}
\begin{proof} Let us use induction with respect to the number of sets: $m.$ Suppose that we have one set ${\mathcal X}_{s}.$ Since $s<N,$ there is a nonzero polynomial $p\in\Pi_n$ such that $p\big\vert_{{\mathcal X}_{s}}=0.$ Now evidently  there is a node $B\notin \mathcal X,$ such that $\rho(A,B)<\epsilon$ and $p(B)\neq 0.$ Thus $p$ is an $n$-fundamental polynomial of the node $B$
with
respect to the set ${\mathcal X}\cup \{B\}.$ Hence, in view of Lemma \ref{XA}, the set ${\mathcal X}_{s}\cup \{B\}$ is $n$-independent. Then, assume that Lemma is true in the case of $m-1$ sets, i.e.,  there is a node $B,$ such that 
$\rho(A,B)<(1/2)\epsilon$ and each set ${\mathcal X}_{s_i}\cup \{B\},\ i=1,\ldots,m-1,$ is $n$-independent. In view of Lemma \ref{eps'}, there is a number $\epsilon'<(1/2)\epsilon$ such that for any $C,$ with $\rho(C,B)<\epsilon',$ each set ${\mathcal X}_{s_i}\cup \{C\},\ i=1,\ldots,m-1,$ is $n$-independent. Next, in view of first step of induction, there is a node $A'$ such that $\rho(A',B)<(1/2)\epsilon$ and the set ${\mathcal X}_{s_m}\cup \{A'\}$ is $n$-independent. Now, it is easily seen that $A'$ is a desirable node.
\end{proof}

Denote the linear space of polynomials of total degree at most $n$
vanishing on ${\mathcal X}$ by
\begin{equation*}{{\mathcal P}}_{n,{\mathcal X}}=\left\{p\in \Pi_n
: p\big\vert_{\mathcal X}=0\right\}.
\end{equation*}
The following two propositions are well-known.
\begin{proposition} [e.g., \cite{HJZ}] \label{PnX} For any node set ${\mathcal X}$ we have that
\begin{equation*}\label{eq:theta1} \dim {{\mathcal P}}_{n,{\mathcal X}} = N - |{\mathcal Y}|,\end{equation*}
where ${\mathcal Y}$ is a maximal $n$-independent subset of ${\mathcal X}.$
\end{proposition}

\begin{proposition}\label{maxline}
Assume that $\ell$ is a line and ${\mathcal X}_{n+1}$ is any subset of
$\ell$ containing $n+1$ points. Then we have that
$$p\in {\Pi_{n}}\quad \text{and} \quad p|_{{\mathcal X}_{n+1}}= 0 \; \Longrightarrow \quad p = \ell r,$$
\end{proposition}
\noindent where $r \in \Pi_{n-1}.$

A plane algebraic curve is the zero set of some bivariate
polynomial. To simplify notation, we shall use the same letter $p$,
say, to denote the polynomial $p$ of degree $\ge 1$ and the curve given by the
equation $p(x,y)=0$. 

\noindent Set $d(n, k) := N_n - N_{n-k} = (1/2) k(2n+3-k).$
The following is a generalization of Proposition \ref{maxline}.
\begin{proposition}[\cite{Raf}, Prop. 3.1]\label{maxcurve}
Let $q$ be an algebraic curve of degree $k \le n$ without multiple
components. Then the following hold.\\
i) Any subset of $q$ containing more than $d(n,k)$ nodes is
$n$-dependent.\\
ii) Any subset ${\mathcal X}_d$ of $q$ containing exactly $d=d(n,k)$
nodes is $n$-independent if and only if the following condition
holds:
\begin{equation}\label{p=qr} p\in {\Pi_{n}}\quad \text{and} \quad p|_{{\mathcal X}_d} = 0 \Longrightarrow  p = qr,
\end{equation}
\end{proposition}
\noindent where $r \in \Pi_{n-k}.$

Thus, according to
Proposition~\ref{maxcurve}, i), at most $d(n,k)$ nodes of $\mathcal
X$ can lie in the curve $q$ of degree $k \le n$.
This motivates the following definition.
\begin{definition}[\cite{Raf}, Def. 3.1]\label{def:maximal}
Given an $n$-independent set of nodes $\mathcal X_s,$ with $s\ge
d(n,k).$ A curve of degree $k \le n$ passing through $d(n,k)$ points
of $\mathcal X_s,$ is called maximal.
\end{definition}

\noindent We say that a node $A$ of an $n$-poised set ${\mathcal X}$ \emph{uses} a line $\ell$ if the latter divides the fundamental polynomial of $A,$ i.e.,
$p^\star_{A} = \ell q,$
for some $q \in \Pi_{n-1}.$

Next, we bring a characterization of maximal curves:

\begin{proposition} [\cite{Raf}, Prop. 3.3] \label{maxcor}  Let a node set ${\mathcal X}$ be n-poised. Then
a polynomial $\mu$ of degree $k,\ k\le n,$ is a maximal curve if and only if it is
used by any node in  ${\mathcal X}\setminus \mu.$
\end{proposition}

\begin{proposition}[\cite{HakTor}, Prop. 3.5] \label{extcurve} Assume that $\sigma$ is an algebraic curve  of degree $k,$ without multiple components, and
${\mathcal X}_s\subset \sigma$ is any $n$-independent node set of cardinality
$s,\ s<d(n,k).$ Then the set ${\mathcal X}_s$ can be extended to a maximal
$n$-independent set ${\mathcal X}_{d}\subset \sigma$ of cardinality $d=d(n,k)$.
\end{proposition}

Finally, let us bring a well-known
\begin{lemma} \label{2cor} Suppose that $m$ linearly independent curves pass through all the nodes of ${\mathcal X}.$
Then for any node $A\notin {\mathcal X}$ there are $m-1$ linearly independent curves, belonging to the linear span of given curves, passing through $A$ and all the nodes of ${\mathcal X}.$
\end{lemma}

\section{\label{4}  Main result}
Let us start with the following result from \cite{HT} (see also, \cite{BHT}).
\begin{theorem}[\cite{HT}, Thm. 1] \label{ht}
Assume that ${\mathcal X}$ is an $n$-independent set of $d(n, k-1)+2$
nodes lying in a curve of degree $k$ with $k\le n.$ Then the curve
is determined uniquely by these nodes. 
\end{theorem}
Next result in this series is the following
\begin{theorem}[\cite{HakTor}, Thm. 4.2]\label{mainth}
Assume that ${\mathcal X}$ is an $n$-independent set of $d(n, k-1)+1$
nodes with $k\le n-1.$
Then two different curves of degree
$k$ pass through all the nodes of ${\mathcal X}$ if and only if all the nodes of
${\mathcal X}$ but one lie in a
maximal curve of degree $k-1.$
\end{theorem}

Now let us present the main result of this paper:
\begin{theorem}\label{mainth}
Assume that ${\mathcal X}$ is an $n$-independent set of $d(n, k-2)+2$
nodes with $k\le n-1.$
Then four linearly independent curves of degree less than or equal to
$k$ pass through all the nodes of ${\mathcal X}$ if and only if all the nodes of
${\mathcal X}$ but two lie in a
maximal curve of degree $k-2.$
\end{theorem}

Let us mention that the inverse implication here is evident. Indeed, assume that $d(n, k-2)$
nodes of ${\mathcal X}$ are located in a curve $\mu$ of degree $k-2$.
Therefore the
curve $\mu$ is maximal and the remaining two nodes of ${\mathcal X},$ denoted by $A$ and $B,$ are outside of it: $A,B\notin \mu.$ Hence we have that
$${\mathcal P}_{k,{\mathcal X}}=\left\{{p : p\in \Pi_{k}}, p(A)=p(B)=0\right\}= \left\{{q\mu : q\in \Pi_{2}}, q(A)=q(B)=0\right\}. $$
Thus we readily get that
$$\dim{\mathcal P}_{k,{\mathcal X}}=\dim\left\{{q\in \Pi_{2}}: q(A)=q(B)=0\right\}=\dim{\mathcal P}_{2, \{A,B\}}=6-2=4. $$
In the last equality we use the fact that any two nodes are $2$-independent. We get also that there can be at most $4$ linearly independent curves of degree $\le k$ passing 
through all the nodes of $\mathcal X.$

Before starting the proof of Theorem \ref{mainth} let us present two lemmas.

\begin{lemma} \label{lem44} Assume that ${\mathcal X}$ is an $n$-independent node set and a node $A\in \mathcal X$ has an $n$-fundamental polynomial $p^\star_A$ such that $p^\star_A(A')\neq 0.$ Then we can replace the node $A$ with $A'$ such that the resulted set ${\mathcal X}':={\mathcal X}\cup\{A'\}\setminus \{A\}$ 
is again an $n$-independent. In particular, such replacement can be done in the following two cases:

i) If a node $A\in \mathcal X$ belongs to several components of $\sigma$ then we can replace it with a node $A',$ which belongs only to one component of $\sigma;$

ii) If a curve $q$ is not a component of an $n$-fundamental polynomial $p^\star_A$ then we can replace the node $A$ with a node $A'$ lying in $q.$
\end{lemma}
 
\begin{proof} Indeed, notice that $p^\star_A(A')\neq 0$ means that $p^\star_A$ is a fundamental polynomial also for the node $A'$ with respect to the set ${\mathcal X}'.$ Next, for (i) note that a fundamental polynomial of a node $A$ differs from $0$ in a neighborhood of $A.$ Finally, for (ii) note that $q$ is not a component of $p^\star_A$
means that there is a point $A'\in q$ such that $p^\star_A(A')\neq 0.$\end{proof}
\begin{lemma}\label{mainl}
Assume that the hypotheses of Theorem \ref{mainth} hold and assume additionally that 
there is a curve $q_{k-1}\in \Pi_{k-1}$ passing through all the nodes of ${\mathcal X}.$
Then all the nodes of
${\mathcal X}$ but two lie in a
maximal curve $\mu$ of degree $k-2.$
\end{lemma}
\begin{proof}
First note that the curve $q_{k-1}$ is of exact degree $k-1,$ since it passes through more than $d(n, k-2)$ $n$-independent nodes. This implies also that $q_{k-1}$ has no multiple component. 
Therefore we can extend the set ${\mathcal X}$ till a maximal $n$-independent set ${\mathcal Y}\subset q_{k-1},$ by adding $n-k+1$ nodes, i.e.,
$${\mathcal Y} = {\mathcal X}\cup{\mathcal A},$$
where ${\mathcal A}=\{A_0,\ldots,A_{n-k}\}.$

In view of Lemma \ref{lem44}, i), we may suppose that the nodes from ${\mathcal A}$ are not intersection points of the components of the curve $q_{k-1}.$

Next, we are going to prove that these $n-k+1$ nodes are collinear together with $m\ge 2$ nodes from ${\mathcal X}.$

To this end denote the line through the nodes $A_0$ and $A_{1}$ by $\ell_{01}.$ Then for each $i=2\ldots,n-k,$ choose a line $\ell_i$ passing through the node $A_{i}$ which is not a component of 
$q_{k-1}.$ We require also that each line passes through
only one of the mentioned nodes and therefore the lines are
distinct. 

Now suppose that $p\in\Pi_k$ vanishes on ${\mathcal X}.$
Consider the polynomial
$r=p\ell_{01}\ell_2\cdots\ell_{n-k}.$ We have that $r\in\Pi_n$ and $r$ vanishes on the node set ${\mathcal Y},$ which is a maximal $n$-independent set in the curve $q_{k-1}.$ Therefore we obtain that $r=q_{k-1}s,$
where $s\in\Pi_{n-k+1}.$ Thus we have that
$$p\ell_{01}\ell_2\cdots\ell_{n-k}=q_{k-1}s.$$
The lines $\ell_i,\ i=2,\ldots,n-k,$ are not components of $q_{k-2}.$ Therefore they are components of the polynomial $s.$ Thus we obtain that
$$p\ell_{01}=q_{k-1}\beta,\ \hbox{where}\ \beta\in\Pi_{2}.$$

Now let us verify that $\ell_{01}$ is a component of $q_{k-1}.$ Indeed, otherwise it is a component of the conic $\beta$ and we get that 
$$p\in \Pi_k, \ p\big\vert_{{\mathcal X}} = 0
 \ \implies\ p=q_{k-1}\ell,\ \hbox{where}\ \ell\in\Pi_{1}.$$
Therefore we get $\dim{\mathcal P}_{k,\mathcal X}=3,$ which contradicts the hypothesis.

Thus we conclude that 
\begin{equation}\label{l01}q_{k-1}=\ell_{01}q_{k-2}\ \hbox{where}\ q_{k-2}\in\Pi_{k-2}.
\end{equation}
Then let us show that all the nodes of ${\mathcal A}$ belong to $\ell_{01}.$
Suppose conversely that a node from ${\mathcal A},$ say $A_2,$ does not belong to the line $\ell_{01}.$ Then in the same way
as in the case of the line $\ell_{01}$ we get that $\ell_{02}$ is a component of $q_{k-1}.$ Thus the node $A_0$ is an intersection point of two components of $q_{k-1},$ i.e., $\ell_{01}$ 
and $\ell_{02},$ which contradicts our assumption.

We have that the curve $q_{k-2}$ passes through at most $d(n,k-2)$ nodes from ${\mathcal X}.$ Hence, in view of the relation \eqref{l01}, we get that at least $2$ nodes from ${\mathcal X}$ belong to the line $\ell_{01}.$

Next we will show that exactly $2$ nodes from ${\mathcal X\cap}\ell_{01}$ do not belong to $ q_{k-2}.$ This will prove Lemma since, in view of \eqref{l01}, we have that ${\mathcal X\setminus}\ell_{01}\subset q_{k-2}$ 

Assume by way of contradiction that at least $3$ nodes  from from ${\mathcal X\cap}\ell_{01}$ do not belong to $ q_{k-2}.$ 

Let us verify that then one can move the node $A_0,$ from $\ell_{01}$ to  a new location $A_0'\in q_{k-2}\setminus \ell_{01},$ such that the resulted set ${\mathcal Y} = {\mathcal X}\cup{\mathcal A}$ remains $n$-independent. This will contradict the proved fact ${\mathcal A}\subset\ell_{01}.$

In view of Lemma \ref{lem44}, ii), for this we need to find an $n$-fundamental polynomial of $A_0$ for which $q_{k-2}$ is not a component. Let us show that  any fundamental polynomial of $A_0$ has  this property. Indeed, suppose conversely that for an $n$-fundamental polynomial $p^\star_{A_0}\in\Pi_n$ the curve $q_{k-2}$ is a component, i.e., 
$p^\star_{A_0}=q_{k-2}r,$ where $r\in \Pi_{n-k+2}.$ We get from here that $r$ vanishes at $\ge 3+(n-k+1)-1=n-k+3$ nodes in $\ell.$ Therefore, in view of Proposition \ref{maxline}, $r$ vanishes at all the points of $\ell_{01}$ including $A_0,$ which is a contradiction.
\end{proof}

Now we are in a position to present
\begin{proof}[Proof of Theorem \ref{mainth}]
Recall that it remains to prove the direct implication. Let  $\sigma_1,\ldots,\sigma_4,$ be the four curves of degree $\le k$ that pass through
all the nodes of the $n$-independent set ${\mathcal X}$ with $|{\mathcal X}| = d(n,
k-2)+2.$ 

Consider three collinear nodes $B_1,B_2,B_3\notin{\mathcal X}$ such that the
following two conditions are satisfied:

$(i)$ The set ${\mathcal X}\cup \{B_1,B_2,B_3\}$ is $n$-independent;

$(ii)$ The line through $B_1,B_2,B_3,$ does not pass through any node from ${\mathcal X}.$

Let us verify that one can find such nodes $B_1,B_2,B_3,$ or the conclusion of Theorem \ref{mainth} holds. Indeed, in view of Lemma \ref{ext}, we can start by choosing some two nodes $B_i',\ i=1,2,$ such that 

$(i)'$ The set $\mathcal X\cup\{B_1’,B_2’\}$ is  $n$-independent. 

Then, according to Lemma \ref{eps'}, for some positive $\epsilon$ all the nodes
in $\epsilon$ neighborhoods of $B_i',\ i=1,2,$ satisfy $(i)'.$
Thus, from this neighborhoods we can choose the nodes $B_i,\ i=1,2,$ respectively, such that the line through them: $\ell_{0}$ does not pass through any node from ${\mathcal X}.$

Now,  let us esablish existence of the node $B_3,$ satisfying the condition $(i).$  

Assume, by way of contradiction, that there is no such node $B_3\in \ell_{0}.$ 

This, in view of Lemma \ref{XA}, means that any polynomial $p\in \Pi_n$ vanishing on ${\mathcal X}\cup \{B_1, B_2,\}$ vanishes identically on $\ell_{0}.$ By using Lemma \ref{2cor} we may choose a such polynomial $p$ from the linear span of linearly independent curves $\sigma_i, i=1,2,3.$ Then we get that $p\in \Pi_{k},\ p\big\vert_{\ell_{0}}=0.$ Thus we have  $p=\ell_{0}q,$ where $q\in\Pi_{k-1}.$
Now, in view of $(ii)$ we readily deduce that the curve  $q$ of degree $\le k-1$ passes through all the nodes of $\mathcal X.$
Thus the proof of Theorem is completed in view of Lemma \ref{mainl}.

Next, in view of Proposition \ref{2cor}, there is a curve of degree at most $k,$ denoted by $\sigma,$ which passes through all
the nodes of ${\mathcal X}':={\mathcal X}\cup \{B_1,B_2,B_3\}.$

Now notice that the curve $\sigma$ passes through more than $d(n,k-2)$ nodes and therefore its degree equals either to $k-1$ or $k.$ By taking into account Lemma \ref{mainl} we may assume that the degree of the curve $\sigma$ equals to $k.$ Evidently we may assume also that $\sigma$ has no multiple component. 

Therefore, by using Proposition \ref{extcurve}, we can extend the set ${\mathcal X}'$ till a maximal
$n$-independent set ${\mathcal X}''\subset\sigma.$ To this end, in view of
$|{\mathcal X}''| = d(n,k)$, we need to add a set of $d(n, k)-(d(n, k-2)+2)-3 = 2(n-k)$
nodes to ${\mathcal X}'$, denoted by $\mathcal A:=\{A_1,\ldots,A_{2(n-k)}\}.$
Thus we have that $${\mathcal X}'':={\mathcal X}\cup \{B_1,B_2,B_3\} \cup \mathcal A.$$

Thus the curve $\sigma$ becomes maximal
with respect to this set. In view of Lemma \ref{lem44}, i), we require that each node of $\mathcal A$ belongs only to one component of the curve $\sigma.$

Then, by using Lemma \ref{2cor}, we get a curve $\sigma_0$ of degree at most $k,$ different from $\sigma$ that passes through all
the nodes of ${\mathcal X}$ and two more nodes:  $A_0,A_0'\in\mathcal A,$ which will be described below.

In the sequel we intend to divide the set of nodes $\mathcal A$ into $n-k$ pairs such that the lines $\ell_1,\ldots,\ell_{n-k-1}$ through the first $n-k-1$ pairs from them, respectively, are not components of $\sigma.$ The last pair is specified as $A_0,A_0'.$  

Before establishing the mentioned division of $\mathcal A$ let us verify how we can finish the proof by using it.

Recall that $\ell_0$ is the line through the triple of the nodes $B_1,B_2,B_3.$ 
Notice that the following polynomial
\begin{equation*}\label{C_2}
\sigma_0 \, \ell_0\, \ell_1 \, \ell_2 \dots \,
\ell_{n-k-1}
\end{equation*}
of degree $n$ vanishes at
all the $d(n,k)$ nodes of ${\mathcal X}''\subset\sigma.$
Consequently, according to Proposition \ref{maxcurve}, $\sigma$ divides this polynomial:
\begin{equation}\label{sigmalq}
\sigma_0 \, \ell_0\, \ell_1 \, \ell_2 \dots \, \ell_{n-k-1}
= \sigma \, q, \quad q \in\Pi_{n-k}.
\end{equation}
The distinct lines  $\ell_1, \ell_2, \dots , \ell_{n-k-1},$ do not
divide the polynomial $\sigma \in\Pi_{k}$, therefore all they have
to divide $q \in\Pi_{n-k}.$ Therefore, we get from \eqref{sigmalq}:
\begin{equation}\label{betabeta'}
\sigma_0 \, \ell_0 = \sigma \, \ell, \ \hbox{where}\ \ell\in\Pi_1.
\end{equation}
Since the curves $\sigma$ and $\sigma_0$ are different the lines $\ell_0$ and $\ell$ also are different.
Therefore the
line $\ell_0$ has to divide $\sigma \in \Pi_k$:
\begin{equation*}
\sigma = \ell_0 \, r, \quad r \in \Pi_{k-1}.
\end{equation*}

In view of above condition (iii) the line $\ell_0$ does not pass through any node of $\mathcal X.$ Therefore the curve  $r$ of degree $k-1$ passes through all the nodes of $\mathcal X.$
Thus the proof of Theorem is completed in view of Lemma \ref{mainl}.

Observe that, in view of the above argument, we may assume that any line component of the curve $\sigma$ passes through at least a node from $\mathcal X.$ 

Next let us establish the above mentioned division of the node set $\mathcal A.$ 

Recall that each node of $\mathcal A$ belongs only to one component of the curve $\sigma.$ Therefore the line components of $\sigma$ do not intersect at the nodes of $\mathcal A.$

By using induction on $m$ one can prove easily the following  
\begin{lemma}\label{lemA0} Suppose that a finite set of lines ${\mathcal L}$ and  $2m$ nodes lying in these lines are given, such that no node is an intersetion point of two lines. Then one can divide the node set into $m$ pairs such that no pair belongs to the same line from ${\mathcal L}$ if and only if each line from ${\mathcal L}$ contains no more than $m$ nodes. 
\end{lemma}
Thus the above mentioned division of the node set $\mathcal A$ into $n-k$ pairs is possible if and only if no $n-k$ nodes of ${\mathcal A}_0:=\mathcal A\setminus \{A_0,A_0'\}$ are located in a line component, where the nodes $A_0,A_0'$ are those associated with the curve $\sigma_0.$ Observe also that we may associate any two nodes of $\mathcal A$ with $\sigma_0.$ 

Now notice that, in view of $|\mathcal A|=2(n-k),$ there can be at most two undesirable line components for the set $\mathcal A,$ i.e., lines containing each $n-k$ nodes from it. This case can be settled easily, by associating with $\sigma_0$ one node from each of the two components.  

Next suppose that there is only one undesirable line component with $n-k$ or $n-k+1$ nodes. In this case one or two nodes from this line we associate with $\sigma_0,$ respectively.

Finally consider the case of one undesirable line component $\ell$ with $s \ge n-k+2$ nodes. We have that 
$$\sigma=\ell q,\ \hbox{where}\ q\in\Pi_{k-1}.$$
 Now we will move $s-n+k-1$ nodes, one by one, from $\ell$ to the other component $q$ such that the set ${\mathcal X}'':={\mathcal X}\cup \{B_1,B_2,B_3\} \cup \mathcal A$ remains $n$-independent. Again, in view of Lemma \ref{lem44}, i), we require that each moved node belongs only to one component of the curve $\sigma.$ 

 To establish each described movement, in view of Lemma \ref{lem44}, ii), it suffices to prove that during this process each node $A\in\ell\cap \mathcal A,$ has no $n$-fundamental polynomial for which the curve $q$ is a component.  Suppose conversely that $p^\star_A=qr, \ r\in \Pi_{n-k+1}.$ Now, we have that $r$ vanishes at $\ge n-k+1$ nodes in $\ell\cap \mathcal A\setminus\{A\},$ and at least at a node from $\ell\cap \mathcal X$ mentioned above. Thus $r$ together with $p^\star_A$ vanishes at the whole line $\ell,$ including the node $A,$
which is a contradiction. It remains to note that there will be no more undesirable line, except $\ell,$ in the resulted set $\mathcal A,$ after the described movement of the nodes, since we finish by keeping exactly   
$n-k+1$ nodes in $\ell\cap \mathcal A.$

\end{proof}

\section{\label{5}  An application to the Gasca-Maeztu conjecture}

Recall that a node $A\in {\mathcal X}$ uses a line $\ell$ means that $\ell$ is a factor of the fundamental polynomial
$p=p^\star_{A},$ i.e., $p = \ell r,$ for some $r \in \Pi_{n-1}.$

A $GC_n$-set in plane is an $n$-poised set of nodes where the fundamental polynomial of each node is a product of $n$ linear factors. 

The Gasca-Maeztu conjecture states that any $GC_n$-set possesses a subset of $n+1$ collinear nodes.

It was proved in \cite{CG} that any line passing through exactly $2$ nodes of a $GC_n$-set ${\mathcal X}$ can be used at most by one
node from ${\mathcal X}.$ 

It was proved in \cite{HT}  that any used  line passing through exactly $3$ nodes of a $GC_n$-set ${\mathcal X}$ can be used either by exactly one
or three nodes from ${\mathcal X}.$ 

Below we consider the case of lines passing through exactly $4$ nodes.
\begin{corollary}
Let ${\mathcal X}$ be an $n$-poised set of nodes and  $\ell$ be a line which passes through exactly $4$ nodes. 
Suppose $\ell$ is used by at least four nodes from ${\mathcal X}.$ Then it is used by exactly six nodes from ${\mathcal X}.$ Moreover, if it is
used by six nodes, then they form a $2$-poised set. Furthermore, in the latter case, if ${\mathcal X}$ is a $GC_n$ set then the six nodes form a $GC_2$ set.
\end{corollary}

\begin{proof}
Assume that $\ell\cap{\mathcal X}=\{A_1,\ldots,A_4\}=:\mathcal A.$ Assume also that the four nodes in $\mathcal B:=\{B_1,\ldots,B_4\}\in {\mathcal X}$ use the line
$\ell:$
$p_{B_i}^{\star}=\ell \, q_i,\ i=1,\ldots,4,$
where $q_i \in \Pi_{n-1}.$

The polynomials $q_1,\ldots, q_4,$ vanish at $N-8$ nodes of the set ${\mathcal X}':={\mathcal X}\setminus(\mathcal A\cup\mathcal B).$ Hence through these $N-8=d(n, n-3)+2$ nodes pass four linearly independent curves of degree $n-1.$ 
By Theorem \ref{mainth} there exists a maximal curve
$\mu$ of degree $n-3$ passing through $N-10$ nodes of ${\mathcal X}'$ and the remaining two nodes denoted by $C_1,C_2,$ are outside of it.
Now, according to Proposition \ref{maxcor}, the nodes $C_1,C_2,$ use $\mu:$
$$p_{C_i}^{\star} = \mu r_i,\ r_i \in \Pi_3,\ i=1,2.$$
These polynomials $r_i$ have to vanish at the four nodes of $\mathcal A\subset\ell.$
Hence $q_i=\ell\beta_i,\ i=1,2,$ with $\beta_i\in \Pi_2.$
Therefore, the nodes $C_1,C_2$ use the line ${\ell}:$
$$p_{C_i}^{\star} = \mu{\ell}\beta_i,\ i=1,2.$$
Hence if four nodes in $\mathcal B\subset{\mathcal X}$ use the line ${\ell}$ then there exist two
more nodes $C_1,C_2\in{\mathcal X}$ using it and all the nodes of ${\mathcal Y}:={\mathcal X}\setminus (\mathcal A\cup\mathcal B\cup\{C_1,C_2\})$ lie in a
maximal curve $\mu$ of degree $n-3:$
\begin{equation}\label{D}{\mathcal Y}\subset\mu.\end{equation}

Next, let us show that there is no
seventh node using ${\ell}$. Assume by way of
contradiction that except of the six nodes in $\mathcal S:=\{B_1,\ldots,B_4, C_1,C_2\},$ there is a seventh node $D$ using ${\ell}.$ Of course we have that
$D\in{\mathcal Y}.$

Then we have that the four nodes $B_1,B_2,B_3$ and $D$ are using $\ell$ therefore, as was proved above, there exist two
more nodes $E_1,E_2\in{\mathcal X}$  (which may coincide or not with $B_4$ or $C_1,C_2$) using it and all the nodes of ${\mathcal Y}':={\mathcal X}\setminus (\mathcal A\cup
\{B_1,B_2,B_3,D, E_1,E_2\})$ lie in a
maximal curve $\mu'$ of degree $n-3.$ 

We have also that
\begin{equation}\label{E}
p_{D}^{\star} = \mu' q', \ q' \in \Pi_3.
\end{equation}

Now, notice that both the curves $\mu$ and $\mu'$ pass through all the nodes of the set
${\mathcal Z}:={\mathcal X}\setminus (\mathcal A\cup\mathcal B\cup\{C_1,C_2,D,E_1,E_2,\})$ with $|{\mathcal Z}|\ge N-13.$

Then, we get from Theorem \ref{ht}, with $k=n-4$, that $N-13 = d(n, n-4)+2$ nodes determine
the curve of degree $n-3$ passing through them uniquely. Thus $\mu$ and $\mu'$ coincide.

Therefore, in view of \eqref{D} and \eqref{E}, $p_{D}^{\star}$ vanishes at all the nodes of ${\mathcal Y}$,
which is a contradiction since $D\in{\mathcal Y}.$

Now, let us verify the last ``moreover" statement. Suppose the six nodes in $\mathcal S\subset{\mathcal X}$ use the line ${\ell}.$ Then, as we obtained earlier, the nodes ${\mathcal Y}:={\mathcal X}\setminus (\mathcal A\cup\mathcal S)$ are located in a
maximal curve $\mu$ of degree $n-3.$ Therefore the fundamental polynomial of each $A\in \mathcal S$ uses $\mu$ and hence $\ell:$
$$\ p^\star_A=\mu \ell q_{A},\ \hbox{where}\ q_{A}\in\Pi_2.$$
It is easily seen that $q_{A}$ is a $2$-fundamental polynomial of $A\in \mathcal S.$
\end{proof}


\end{document}